\documentclass{scrartcl}
\usepackage[top=30truemm,bottom=40truemm,left=35truemm,right=35truemm]{geometry}
\usepackage{amssymb}
\usepackage{amsmath}
\usepackage{amsthm}
\usepackage{mathrsfs} 
\usepackage{amsfonts}
\usepackage[dvipdfmx]{graphicx}
\usepackage{bm}
\usepackage{color}
\usepackage{url}

\definecolor{mygreen}{cmyk}{0.64,0.00,0.95,0.40}

\newcommand{\ep}{\varepsilon}
\newcommand{\dual}[1]{\langle{#1}\rangle}

\title{
Convergence of the immersed-boundary\\ finite-element method for the \\Stokes problem%
}
\author{
Norikazu \textsc{Saito}%
\thanks{Graduate School of Mathematical Sciences, The University of Tokyo,
Komaba 3-8-1, Meguro-ku, Tokyo 153-8914, Japan.
\textit{E-mail}: \texttt{norikazu@ms.u-tokyo.ac.jp}}
\and
Yoshiki \textsc{Sugitani}%
\thanks{Graduate School of Mathematical Sciences, The University of Tokyo,
Komaba 3-8-1, Meguro-ku, Tokyo 153-8914, Japan. 
\textit{E-mail}: \texttt{sugitani@ms.u-tokyo.ac.jp}}
}

\begin{document}

\theoremstyle{plain}
\newtheorem{thm}{Theorem}
\newtheorem{prop}[thm]{Proposition}
\newtheorem{cor}[thm]{Corollary}
\newtheorem{lemma}[thm]{Lemma}
\theoremstyle{remark}
\newtheorem{defi}[thm]{Definition}
\newtheorem{remark}[thm]{Remark}
\newtheorem{assum}[thm]{Assumption}
\newtheorem{ex}[thm]{Example}


\maketitle

 \begin{abstract}
Convergence results for the immersed boundary method applied to a model
  Stokes problem with the homogeneous Dirichlet boundary condition are presented. As a discretization method, we deal with the finite element method. First, the
  immersed force field is approximated using a regularized delta
  function and its error in the $W^{-1,p}$ norm is examined for $1\le
  p<n/(n-1)$, $n$ being the space dimension. Then, we
  consider the immersed boundary discretization of the Stokes problem
  and study the
  regularization and discretization errors separately. Consequently,
  error estimate of order $h^{1-\alpha}$ in the $W^{1,1}\times L^1$ norm
  for the velocity
  and pressure is derived, where $\alpha$ is an arbitrarily small positive
  number. Error estimate of order $h^{1-\alpha}$ in the $L^r$ norm
  for the velocity is also derived with $r=n/(n-1-\alpha)$. The validity of those theoretical results
are confirmed by numerical examples. 
 \end{abstract}


\bigskip


\section{Introduction}
\label{sec:intro}

The immersed boundary (IB) method is a powerful method for solving a class of
{fluid-structure interaction} problems originally proposed by
Peskin \cite{pes72,pes77} to simulate the blood flow through artificial
heart valves. For later developments, see \cite{pes02}. The IB
method is also successfully applied to multi-phase flow
problems, elliptic interface problems, and so on. 

In contrast to a huge number of applications,
it seems that there are only a few results about {theoretical
convergence analysis}. The pioneering work was done by Y. Mori in 2008 (see
\cite{mor08}). He studied a model (stationary) Stokes problem for the
velocity $u$ and pressure $q$ in an $n$ dimensional
torus $U=[\mathbb{R}/(2\pi\mathbb{Z})]^n\subset\mathbb{R}^n$, 
\setcounter{equation}{-1}
\begin{equation}
\label{eq:mori1}
 - \Delta u +\nabla q  = f-g \mbox{ in }U, \quad
 \nabla \cdot u =0 \mbox{ in } U,
\end{equation}
with
\[
 f(x)=\int_{\Xi} F(\theta)\delta(x-X(\theta))~d\theta,\quad
 g=\frac{1}{(2\pi)^n}\int_\Theta F(\theta)~d\theta.
\]
Herein, the \emph{immersed boundary} $\Gamma\subset U$, which is assumed
to be a hypersurface of $\mathbb{R}^n$, is parameterizaed as
\[
 \Gamma=\{X(\theta)=(X_1(\theta),\ldots,X_n(\theta)) \mid \theta\in
 \Theta\},
\]
where $\Theta$ denotes a subset of $\mathbb{R}^{n-1}$; see Figure \ref{fig:domain0}. The function $F=F(\theta)$ denotes the force
distributed along $\Gamma$ and $\delta=\delta(x)$ the (scalar-valued) Dirac delta
function. (In \cite{mor08}, the case $n=2$ was explicitly mentioned.)
Introducing the regularized delta function $\delta^h\approx \delta$ with a parameter $h>0$, he
considered the regularized Stokes problem
\begin{equation*}
 - \Delta \tilde{u} +\nabla \tilde{q}  = \int_{\Xi} F(\theta)\delta^h(x-X(\theta))~d\theta-g \mbox{ in }U, \quad
 \nabla \cdot \tilde{u} =0 \mbox{ in } U.
\end{equation*}
The regularized problem was discretized by the finite difference
method using a uniform Eulerian grid with grid size $h$. Then, he
succeeded in deriving the maximum norm error estimate for the velocity of the form
\[
 \|u-\tilde{u}_h\|_{L^\infty(U)}\le C(h+h^\alpha)|\log h|\quad
 (\alpha>0\mbox{ suitable constant})
\]
under regularity assumptions on $\Gamma$ and $F$ together with
structural assumptions on $\delta^h$. Herein, $\tilde{u}_h$ denotes the finite difference
solution. After that, the
method and results were extended to several directions (see
\cite{lm12,lm14}). For example, several $L^p$-error estimates, $1\le
p\le\infty$, were obtained in \cite{lm14}. A typical result is given as  
\[
 \|u-\tilde{u}_h\|_{L^p(U)}+ h\|q-\tilde{q}_h\|_{L^p(U)}\le
 Ch^2|\log h|^\eta\quad 
 (\eta>0\mbox{ suitable constant}).
\]
Similar results for the Poisson interface
problem was presented in \cite{li15}. On the other hand, we observe from
numerical experiments that the IB method has a first order accuracy for the velocity in the
$L^\infty$ norm. Therefore, those estimates
are only sub-optimal and the proof of optimal-order error estimate is
still open at present. Moreover, the explicit formula of the Green function
associated with \eqref{eq:mori1} was used to derive error estimates in
\cite{lm12,lm14,mor08}. Hence, it is difficult to apply those methods to
more standard settings, for example, to the Dirichlet boundary value problem.

 \begin{figure}[ht]
  \begin{center}
   \includegraphics[width=.4\textwidth]{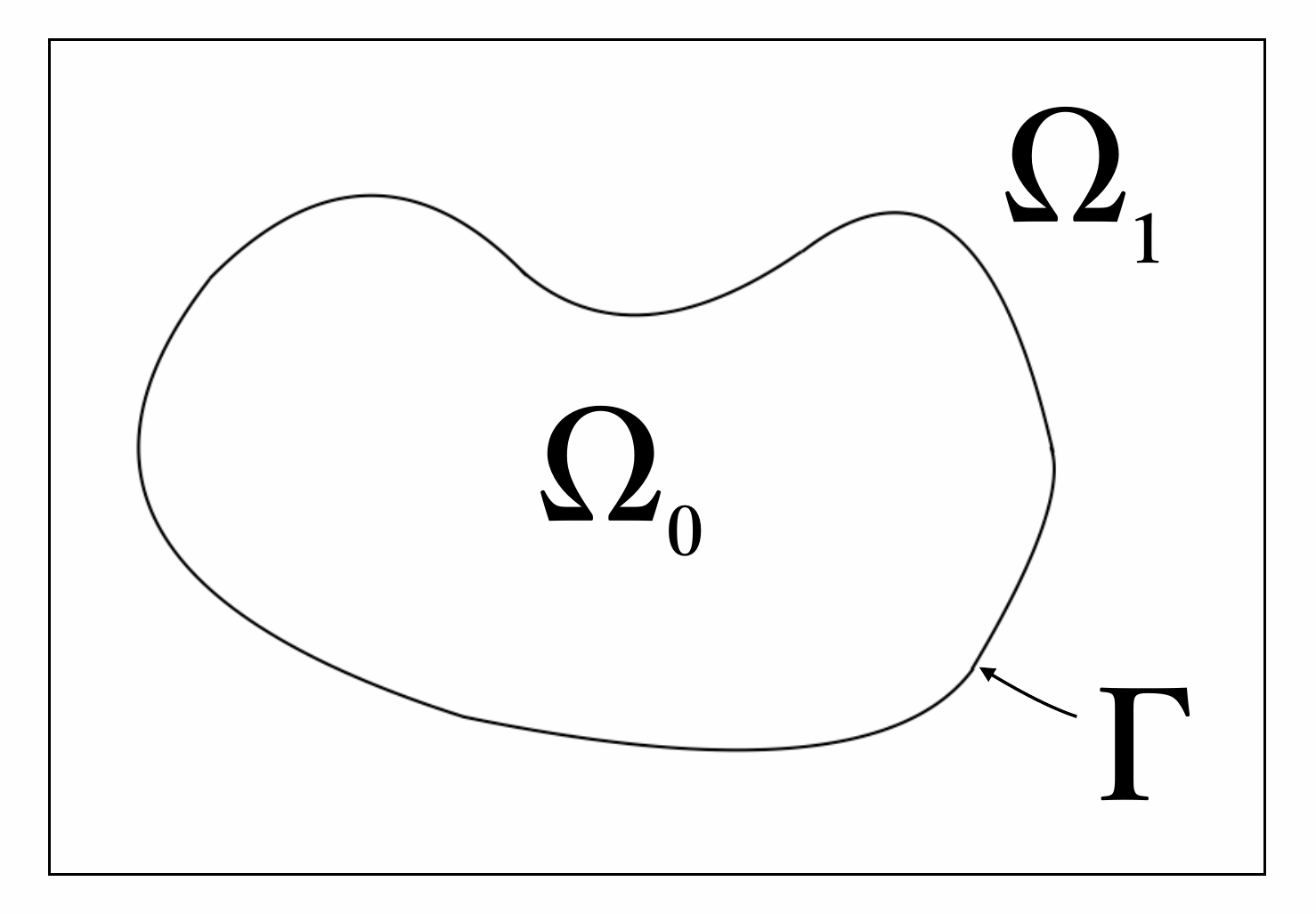}
 \end{center}
\caption{$\overline{\Omega}=\overline{\Omega_0\cup\Omega_1}$, $\Omega_0$, $\Omega_1$ and $\Gamma$.}
\label{fig:domain0}
 \end{figure}

In this paper, we take a different approach. We consider the
Dirichlet boundary value problem for the Stokes equations \eqref{eq:1}
below and study the
\emph{regularization error and discretized error
separately} in Sections \ref{sec:2} and \ref{sec:6}. To this end, we first give interpretations of the
immersed outer force $f$ above as an $\mathbb{R}^n$-Lebesgue measure and as a
functional over $W^{1,p}_0(\Omega)^n$; see
Propositions \ref{lemma:f1} and \ref{la:f-int}. (The
meaning of mathematical symbols will be mentioned in Paragraph \ref{sec:2.0}.) 
Then, we introduce a regularized delta function
$\delta^\ep$ with a parameter $\ep>0$ and examine the error between $f$ and its regularization
\[
 f^\ep(x)= \int_{\Xi} F(\theta)\delta^\ep(x-X(\theta))~d\theta
\]
in the $W^{-1,p}(\Omega)^n$ norm for $1\le p<\frac{n}{n-1}$; see Proposition
\ref{la:reg-err}. Estimate for the
regularization error (see Proposition \ref{la:8}) is a direct consequence of Proposition
\ref{la:reg-err} and the stability result of \cite{mr07} (or (A1$_p$)
below). After introducing structural assumptions on $\delta^\ep$,
\[
 \delta^\ep (x) = \frac1{\ep^n} \prod_{i=1}^n \phi\left(\frac{x_i}\ep\right),
\]
that is
essentially the same as that of \cite{lm12,lm14,mor08}, we show that the
$W^{1,p}\times L^p$ error estimate for the velocity and pressure is of
order $\ep^{1-n+\frac{n}{p}}$ if $1\le p<\frac{n}{n-1}$; see Proposition
\ref{la:err0}.

Then, we proceed to the study of discretization in Section \ref{sec:6}
. We are concerned with the finite
element method rather than the finite difference method. This enable us
to apply several sharp $W^{1,p}\times L^p$ stability
and error estimates due to \cite{gns15} (or (A2$_p$) below). Finally, we
obtain several (still sub-optimal but nearly-optimal) error estimates in
several norms; see Theorem \ref{thm:error} which is the main result of
this paper. The effect of numerical integration for computing $f^\ep$ is
discussed in Section \ref{sec:7}.
Actually, a simple numerical integration formula does not spoil the
accuracy of the IB method (see Proposition \ref{la:reg-err2} and Theorem \ref{thm:error1}).
The validity of those theoretical results
are confirmed by numerical examples in Section \ref{sec:8}. 

We only assume that 
$\phi$ is a continuous function in $\mathbb{R}$ with compact support and
with the unit mean value (see \eqref{eq:del}). 
{On the other hand, several conditions on moment and smoothing orders of
$\phi$ were assumed in \cite{lm12,lm14,mor08}; we are able to remove those
restrictions.   }

It should be kept in mind that our aim is to reveal the accuracy of the
regularization and discretization procedures and is not to propose a new
computational method; see also Remark \ref{rem:15}. We consider the
finite element method only as a model discretization method.

\section{Immersed boundary formulation}
\label{sec:2}

\subsection{Geometry and notation}
\label{sec:2.0}

Suppose that $\Omega$ is a polyhedral domain in $\mathbb{R}^n$, 
$n=2,3$, with the boundary $\partial\Omega$.
The domain $\Omega$ is
divided into two disjoint components $\Omega_0$ and $\Omega_1$ by a
simple closed curve ($n=2$) or surface ($n=3$) which is designated by
$\Gamma$. The curve (surface) $\Gamma$ is called
the \emph{immersed boundary} and is supposed to be parametrized as $\Gamma=\{{X}(\theta)=(X_1(\theta),\ldots,X_n(\theta))\mid \theta\in\Theta\}$ where
$\Theta$ is a bounded subset of $\mathbb{R}^{n-1}$ for the Lagrangian coordinate. See Fig. \ref{fig:domain0} for example.
We set 
\[
 J_{{X}}(\theta)=
  \begin{cases}
   \sqrt{\left|\frac{\partial{X}_1}{\partial \theta}\right|^2+\left|\frac{\partial{X}_2}{\partial \theta}\right|^2} &
   \mbox{if }n=2,\\
   \sqrt{
   \left|\frac{\partial({X}_2,{X}_3)}{\partial
   (\theta_1,\theta_2)}\right|^2
+\left|\frac{\partial({X}_3,{X}_1)}{\partial
   (\theta_1,\theta_2)}\right|^2
+\left|\frac{\partial({X}_1,{X}_2)}{\partial
   (\theta_1,\theta_2)}\right|^2}&  \mbox{if }n=3.
  \end{cases}
\]
Throughout this paper, we assume the following:  
 \begin{itemize}
  \item $\Gamma$ is a $C^1$ boundary ($X(\theta)$ is
	     a $C^1$ function); 
  \item $\operatorname{dist}(\Gamma,\partial\Omega)>0;$
  \item $J_{{X}}(\theta)\ne 0$ \  $(\theta\in\Theta)$.
 \end{itemize}

 We collect here the notation used in this paper.
 We follow the notation of \cite{af03} for function spaces and their
 norms. For a function space $X$, the
space $X^n$
stands for a product space $X\times \cdots \times X$. For abbreviations, we write as,
for example, 
\[
 \|u\|_{W^{1,p}}= \|u\|_{W^{1,p}(\Omega)^n},\qquad \|\pi\|_{L^p}=\|\pi\|_{L^p(\Omega)}.
\]
We set $W^{1,p}_0(\Omega)=\{v\in W^{1,p}(\Omega)\mid v|_{\partial\Omega}=0\}$ and $W^{-1,p}(\Omega)$ the
topological dual of $W^{1,p}_0(\Omega)$. 
The dual product between
$W^{-1,p}(\Omega)^n$ and $W_0^{1,p}(\Omega)^n$ is denoted by
$\dual{\cdot,\cdot}_{W^{-1,p},W_0^{1,p}}$.
We let $L^p_0(\Omega)=\{q\in L^p(\Omega)\mid \int_\Omega q~dx=0\}$. 
Set 
$B(a,r)=\{x\in\mathbb{R}^n\mid |x-a|<r\}$ for $a\in\mathbb{R}^n$ and $r>0$. 

For $1\le p\le \infty$, let $p'$ be the conjugate exponent of $p$; $1\le
p'\le
\infty$ and $\frac{1}{p}+\frac{1}{p'}=1$. 

For vectors $a=(a_1,\ldots,a_n),~b=(b_1,\ldots,b_n)\in\mathbb{R}^n$,
let us denote by $a\cdot b=a_1b_1+\cdots+a_nb_n$ the scalar product. 

\subsection{Immersed boundary force}
\label{sec:2.2}

We set ({formally at this stage}) the immersed boundary force field 
$f:\Omega\to\mathbb{R}^n$ as 
   \begin{equation}
    \label{eq:f0}
     f = \int_\Theta  {F}(\theta) \delta_{{X}(\theta)} ~d\theta 
   \end{equation}
   for $F\in L^1(\Theta)^n$. Hereinafter, we set
   $\delta_a(x)=\delta(x-a)$ for $a\in{\mathbb R}^n$. We have (still {formally})
   \begin{equation}
    \label{eq:f-int0}
     \int_\Omega f(x)\cdot \varphi(x)~dx
     = \int_\Theta F(\theta) \cdot \varphi({X}(\theta))~d\theta 
\qquad (\varphi\in C_0^\infty(\Omega)^n).
   \end{equation}

We state two interpretations of \eqref{eq:f-int0}.  

\begin{prop}
 \label{lemma:f1}
 Let $F\in L^1(\Theta)^n$. 
Then, $f$ defined as \eqref{eq:f0} is 
a finitely signed measure on $\Omega$, with which the integration is defined for any (vector-valued)
measurable function $\varphi$ on $\Omega$. 
In particular, if $F\in L^p(\Theta)^n$ for $1\le p\le \infty$, the integrant is given by  
\begin{equation*}
\dual{f,\varphi}=\int_\Omega \varphi~df = \int_\Theta F(\theta) \cdot \varphi({X}(\theta))~d\theta
\end{equation*}
for any $\varphi\in W^{1,p'}(\Omega)^n$.  
Moreover, $f$ is a singular measure against the Lebesgue measure on
 ${\Omega}$ and, consequently, $f\notin L^1(\Omega)^n$. 
\end{prop}

\begin{proof}
We identify $\delta_a(x)=\delta(x-a)$ with the Dirac measure concentrated at $a\in {\mathbb R}^n$. 
Then, for any measurable set $B\subset{\mathbb R}^n$ and $\theta\in\Theta$, we have
\begin{equation*}
 \delta_{X(\theta)} (B) 
= 1_{X^{-1}(B)} (\theta)
=
\begin{cases}
1 &(X(\theta)\in B)\\
0 &(X(\theta)\notin B),
\end{cases}
\end{equation*}
where $1_{X^{-1}(B)}$ denotes the indicator function of $X^{-1}(B)$ on $\Theta$. 
By virtue of Lebesgue's dominated convergence theorem, we derive for any disjoint measurable sets $\{B_n\}_n$ 
\begin{align*}
f\left(\bigcup_{n=1}^\infty B_n\right) 
&=
\int_\Theta F(\theta) \sum_{n=1}^\infty\delta_{{X}(\theta)}(B_n)~d\theta 
= \int_\Theta F(\theta) \sum_{n=1}^\infty 1_{{X}^{-1}(B_n)} (\theta) ~d\theta \\
&= \sum_{n=1}^\infty \int_\Theta F(\theta) 1_{{X}^{-1}(B_n)} (\theta) ~d\theta= 
\sum_{n=1}^\infty \int_\Theta F(\theta) \delta_{X(\theta)}(B_n)
 ~d\theta \\
 &= \sum_{n=1}^\infty f(B_n). 
\end{align*}
Herein, note that $F(\theta)\sum_{n=1}^N 1_{{X}^{-1}(B_n)} (\theta)$ is integrable for any 
$N\in {\mathbb N}$ since $F\in L^1(\Theta)^n$ and $B_n$ is disjoint.
It follows $f(\emptyset)=0$ from $\delta_a(\emptyset)=0$ for all 
$a\in {\mathbb R}^n$. Thus, $f$ is a finitely signed measure on $\Omega$ so that 
the integral $\int_\Omega \varphi~df$ is well-defined for all measurable function $\varphi$. 
According to an integral with the Dirac measure, we have
\begin{equation*}
 \int_\Omega \varphi~df = \int_\Theta F(\theta) \cdot \varphi({X}(\theta))~d\theta,
\end{equation*}
where the right hand side is meaningful for $F\in L^p(\Theta)^n$ and $\varphi \in W^{1,p'}(\Omega)^n$. 
Although the ${\mathbb R}^n$-Lebesgue measure 
$m(\Gamma)$ of $\Gamma$ vanishes (note that $\Gamma$ is ``very thin''), we have $f
 (\Gamma) \neq 0$. Hence, 
$f$ is singular against $m$. 
Finally, the fact $f\notin L^1(\Omega)$ follows from the Lebesgue 
decomposition theorem. 
\end{proof}

Although $f$ does not belong to any $L^p(\Omega)$ spaces as is mentioned
in Proposition \ref{lemma:f1}, it is well-defined as
a functional on $W^{1,p}(\Omega)^n$. 

\begin{prop}
\label{la:f-int}
 Let $1\le p <\infty$ and $F\in L^p(\Theta)^n$.
 Then, the functional
\begin{equation*}
\dual{f,\varphi} = \int_\Theta F(\theta) \cdot \varphi({X}(\theta))~d\theta 
 \quad (\varphi\in C_0^\infty(\Omega)^n)
\end{equation*}
is extended by continuity to a bounded linear functional on $W^{1,p}_0(\Omega)^n$, which
 will be denoted by $\dual{f,\cdot}_{W^{-1,p},W_0^{1,p}}$ below. That
 is, we have $f\in W^{-1, p}(\Omega)^n$.
\end{prop}

\begin{proof}
Let $\varphi\in C_0^\infty(\Omega)^n$. Since
\begin{equation*}
\int_\Gamma \mid\varphi |_\Gamma\mid^{p'}~d\Gamma 
= \int_\Theta \mid \varphi({X}(\theta))\mid^{p'} |J_X(\theta)| ~d\theta,
\end{equation*}
we have by the trace theorem
\begin{align*}
 \dual{f,\varphi}
&\le \|F\|_{L^p(\Theta)} \left( 
\int_\Theta \mid \varphi({X }(\theta))\mid^{p'}~d\theta\right)^\frac1{p'} \\
&\le \|F\|_{L^p(\Theta)}\|J_X\|_{L^\infty(\Theta)}^{-\frac1{p'}} 
\|\varphi\|_{L^{p'}(\Gamma)} 
\le  C \|F\|_{L^p(\Theta)}\|\varphi\|_{W^{1,{p'}}(\Omega)}.
\end{align*}
\end{proof}

Let $\ep>0$ be a regularized parameter. Take a continuous function
$\delta^\ep=\delta^\ep(x)$ satisfying
\begin{equation}
 \label{eq:dep}
  \operatorname{supp}\delta^\ep \subset B(0,K\ep)
\end{equation}
with $K>0$.

Setting $\delta_a^\ep(x)=\delta^\ep(x-a)$ for $a\in\mathbb{R}^n$, we introduce the \emph{regularized
immersed force field} as 
\begin{equation}
 \label{eq:fe}
 f^\ep = \int_\Theta   F(\theta) \delta^\ep_{{X}(\theta)} ~d\theta .
\end{equation}

Since $\delta^\ep_a\in L^\infty(\Omega)$, we have $f^\ep\in
L^\infty(\Omega)$ for $F\in L^1(\Theta)$. The following result plays the most crucial role in this study. 

 \begin{prop}
  \label{la:reg-err}
Suppose that we are given a continuous function $\delta^\ep$ satisfying
  \eqref{eq:dep}. Then, for $1\le p<\frac{n}{n-1}$ and $F\in L^p(\Theta)$, we have
\begin{equation*}
\| f-f^\ep \|_{W^{-1,p}} \le C_0 \|F\|_{L^p(\Theta)}
\left[ \left| 1- \int_{{\mathbb R}^n} \delta^\ep (y) ~dy \right|
+  \| \rho \delta^\ep\|_{L^p({\mathbb R}^n)} 
 \right],
\end{equation*}
where $\rho(x)=x$ and $C_0$ denotes a positive constant depending only
  on $n$, $p$ and $\|J_X\|_{L^\infty(\Theta)}$. 
 \end{prop}

  \begin{proof}
  Let $\varphi \in C^\infty_0(\Omega)^n$ and express it as
  \[
  \varphi(x) = \varphi({X}(\theta)) + (x-{X}(\theta))\cdot 
  \int_0^1 \nabla \varphi( t( x-{X}(\theta) ) + {X}(\theta) ) ~dt
  \quad (x\in\mathbb{R}^n).
  \]
Then, applying Fubini's lemma, we have 
\begin{multline*}
  \dual{f-f^\ep, \varphi}
=\underbrace{\int_\Theta F(\theta) \varphi(X(\theta))
\left( 1- \int_\Omega \delta_{X(\theta)}^\ep (x) ~dx \right) ~d\theta}_{=I_1}  \\
\underbrace{-\int_0^1\int_\Theta F(\theta) \int_\Omega  \delta_{X(\theta)}^\ep (x)
(x-X(\theta))\cdot\nabla \varphi( t( x-X(\theta) ) + X(\theta) ) ~dxd\theta dt}_{=I_2}. 
\end{multline*}

For a sufficiently small $\ep$, we have  $B(X(\theta), K\ep)\subset \Omega$ and 
\[
 \int_\Omega \delta_{X(\theta)}^\ep ~dx = 
\int_{B(X(\theta), K\ep)} \delta_{X(\theta)}^\ep (x)~dx = 
\int_{B(0, K\ep)} \delta^\ep(y) ~dy = 
\int_{{\mathbb R }^n} \delta^\ep(y) ~dy  .
\] 
Hence,  
\begin{align*}
|I_1|&\le \left| 1- \int_{{\mathbb R}^n} \delta^\ep (y) ~dy \right|\int_\Theta |F(\theta)|\cdot |\varphi(X(\theta))|~d\theta \\
&\le \|F\|_{L^p(\Theta)} \left( \int_\Theta |\varphi(X(\theta))|^{p'} ~d\theta \right)^\frac1{p'}
\left| 1- \int_{{\mathbb R}^n} \delta^\ep (y) ~dy \right| \\
&\le \frac{ \|F\|_{L^p(\Theta)}}{\|J_X\|_{L^\infty(\Theta)}^{1/p'}} \|\varphi\|_{L^{p'}(\Gamma)} 
\left| 1- \int_{{\mathbb R}^n} \delta^\ep (y) ~dy \right| \\
&\le C \|F\|_{L^p(\Theta)}\left| 1- \int_{{\mathbb R}^n} \delta^\ep (y) ~dy \right| 
\|\varphi\|_{W^{1,{p'}}(\Omega)}.
\end{align*}

By virtue of H\"older's inequality, we have
\begin{align*}
|I_2|
 &\le \int_0^1\int_\Theta |F(\theta)|\cdot  
 \|  (x-X(\theta)) \delta_{X(\theta)}^\ep \|_{L^p(\Omega)}\cdot \\
 &\mbox{ }\qquad \cdot
\left[ \int_\Omega |\nabla \varphi( t( x-X(\theta) ) + X(\theta) ) |^{p'} ~dx \right]^\frac1{p'}
 ~d\theta dt \\
 &\le \| \rho\delta^\ep\|_{L^p({\mathbb R}^n)} \int_0^1\int_\Theta |F(\theta)| 
 \left[ \int_{{\mathbb R}^n} |\nabla \tilde{\varphi}( t( x-X(\theta) ) + X(\theta) ) |^{p'} ~dx
 \right] ^\frac1{p'} ~d\theta \\
 &\le \| \rho\delta^\ep \|_{L^p({\mathbb R}^n)} \int_0^1\int_\Theta |F(\theta)| ~d\theta
 \left[ \frac1{t^n} \int_{{\mathbb R}^n} |\nabla \tilde{\varphi}(z) |^{p'} ~dz \right]^\frac1{p'} \\
 &\le \| \rho\delta^\ep \|_{L^p({\mathbb R}^n)} \|F\|_{L^1(\Theta)}
 \left(\int_0^1 t^{-\frac n{p'}}~dt\right)
 \|\tilde{\varphi}\|_{W^{1,{p'}}({\mathbb R}^n)}\\
 &\le \frac{p'}{p'-n}\|F\|_{L^1(\Theta)} \| \rho\delta^\ep \|_{L^p({\mathbb R}^n)}\|{\varphi}\|_{W^{1,{p'}}},
\end{align*}
where $\tilde\varphi$ denotes the zero extension of $\varphi$ into
  $\mathbb{R}^n$ and $z = t( x-X(\theta) ) + X(\theta)$. 
   (Note that $n<p'\le \infty$ by $1\le p<\frac{n}{n-1}$.)
  \end{proof}

\begin{remark}
We take $\varphi_0 \in C^\infty_0(\Omega)$ satisfying 
\begin{equation*}
\varphi_0(x)=1 \text{ if }x \in \Gamma(\ep) 
= \{x\in\Omega\mid\text{dist}(x, \Gamma) < \ep \} \cup \Omega_0. 
\end{equation*}
Then, $I_2$ in the proof above vanishes and  
\[
 \|f-f^\ep\|_{W^{-1,p}}\ge\frac{\dual{f-f^\ep, \varphi_0}}{\|\varphi_0\|_{W^{1, p'}}} 
= \frac1{\|\varphi_0\|_{W^{1,p'}}}\int_\Theta F(\theta) ~d\theta
\left[ 1- \int_{{\mathbb R}^n} \delta^\ep (x) ~dx \right] .
\]
 Hence,
 \[
   \int_{{\mathbb R}^n} \delta^\ep (x) ~dx\to 1\quad (\ep\to 0)
 \]
 is a necessary condition for 
 $\|f-f^\ep\|_{W^{-1,p}}\to 0$ to hold.
\label{rem:2.7} 
\end{remark}

  \subsection{Target and regularized problems}
\label{sec:2.3}

We proceed to the formulation of the immersed boundary method.
Using $f$ and $f^\ep$ defined by \eqref{eq:f0} and \eqref{eq:fe},
we consider, respectively, the immersed boundary formulation to the
Stokes equations for the velocity $u$ and pressure $\pi$, 
\begin{equation}
\label{eq:1}
 -\nu \Delta u +\nabla \pi  = f \mbox{ in }\Omega, \quad
 \nabla \cdot u =0 \mbox{ in } \Omega,
 \quad u=0 \mbox{ on }  \partial \Omega
\end{equation}
and its regularized problem for $u^\ep$ and $\pi^\ep$, 
\begin{equation}
\label{eq:1e}
 -\nu \Delta u^\ep +\nabla \pi^\ep  = f^\ep \mbox{ in }\Omega, \quad
 \nabla \cdot u^\ep =0 \mbox{ in } \Omega,
 \quad u^\ep=0 \mbox{ on }  \partial \Omega.
\end{equation}

By a \emph{weak solution} $(u,\pi)\in W_0^{1,p}(\Omega)^n\times
L^p_0(\Omega)$ of \eqref{eq:1} for example, we
mean a solution of the following variational equations: Find $(u,\pi)\in
W_0^{1,p}(\Omega)^n\times L_0^p(\Omega)$ such that  
\begin{subequations}
\label{eq:vp0}
\begin{align}
 & a(u,v) + b(\pi,v) =\dual{f,v}_{W^{-1,p},W_0^{1,p}}
  && (\forall v \in W^{1, p'}_0(\Omega)^n), \label{eq:vp01}\\
&b(q,u) = 0 && (\forall q \in L^{p'}_0(\Omega)), \label{eq:vp02}
\end{align}
\end{subequations}
where
\begin{subequations}
\label{eq:forms}
\begin{align}
a(u,v)&=\frac{\nu}{2}\int_{\Omega} \left( \frac{\partial u_j}{\partial
				   x_i} + \frac{\partial u_i}{\partial
				   x_j} \right)\left( \frac{\partial
				   v_j}{\partial x_i} + \frac{\partial
				   v_i}{\partial x_j} \right)~dx,\label{eq:a}\\ 
b(\pi,u)&=-\int_{\Omega} \pi (\nabla\cdot u)~dx.\label{eq:b}
\end{align}
\end{subequations}

\begin{remark}
Problem \eqref{eq:vp0} can be directly discretized by the finite element
 method with no regularization of $f$. Such methods were studied in
 \cite{bgh07,tab07} for nonstationary Navier-Stokes equations. However,
 our aim here is to reveal the accuracy of the
regularization and discretization procedures as is mentioned in Introduction. 
 \label{rem:15}
\end{remark}

\begin{remark}
 The bilinear form $a$ defined by \eqref{eq:a} is based on the
 deformation-rate tensor $[(1/2)(u_{i,j}+u_{j,i})]_{1\le i,j\le
 n}$. Another definition
 \[
  a(u,v)={\nu}\int_{\Omega} \frac{\partial u_j}{\partial  x_i}\frac{\partial
				   v_j}{\partial x_i}~dx
 \]
 is also available. 
 However, with \eqref{eq:a}, our problem is (essentially) equivalent to
 a two-phase Stokes problem considered in \cite{tab07} for example. 
\end{remark}

We make the following assumption for $1\le p<\infty$:

\smallskip

\noindent \textbf{(A1$_p$)} 
For a given $g\in W^{-1,p}(\Omega)^n$, there exists a unique weak solution $(w,r)\in W^{1,p}_0(\Omega)\times 
L_0^p(\Omega)$ of the Stokes problem, 
\begin{equation}
 \label{eq:w}
  -\nu \Delta w +\nabla r  = g \mbox{ in }\Omega, \quad
 \nabla \cdot w =0 \mbox{ in } \Omega,
 \quad w=0 \mbox{ on }  \partial \Omega
\end{equation}
satisfying 
\begin{equation}
 \label{eq:11z}
\|w\|_{W^{1,p}}+\|r\|_{L^p}\le C_1 \|g\|_{W^{-1,p}}. 
\end{equation}
	    Moreover, if $g\in W^{-1,2}(\Omega)^n\cap
	    L^p(\Omega)$, we have $(w,r)\in W^{2,p}(\Omega)\times 
	    W^{1,p}(\Omega)$ and
	    \begin{equation}
 \label{eq:1a1z}
\|w\|_{W^{2,p}}+\|r\|_{W^{1,p}}\le C_2 \|g\|_{L^p}.
	    \end{equation}
	    Herein, $C_1$ and $C_2$ denote positive constants depending only on
	    $p$ and $\Omega$.

 \begin{remark}
 If $\Omega$ is a convex Lipschitz domain and $1<p\le 2$, (A1$_p$) is
  satisfied in view of \cite[Example 5.5]{mr07} and Lemma
  \ref{la:f-int}. However, we
  directly assume (A1$_p$) instead of the shape condition on
  $\Omega$. Below, $p$ will be restricted as $p<n/(n-1)$.   
 \label{rem:2.10}
 \end{remark}

 The following result is a direct consequence of Lemma \ref{la:reg-err}
 and (A1$_p$).

  \begin{prop}
   \label{la:8}
  Let $1\le p<\frac{n}{n-1}$ and suppose that {\rm (A1$_p$)} is
  satisfied. Let $F\in L^p(\Theta)$. 
Let $(u,\pi)$ and $(u^\ep,\pi^\ep)$ be the weak solutions of \eqref{eq:1}
 and \eqref{eq:1e}, respectively. Then, we have
\begin{equation*}
  \| u-u^\ep \|_{W^{1,p}} + \|\pi-\pi^\ep\|_{L^p} 
  \le C_0C_1\|F\|_{L^p(\Theta)}
\left[ \left| 1- \int_{{\mathbb R}^n} \delta^\ep (y) ~dy \right|
+  \| \rho \delta^\ep\|_{L^p({\mathbb R}^n)} 
 \right]. 
\end{equation*}
  \end{prop}

%
The most familiar choice of $\delta^\ep$ is given by a product 
of one variable functions: 
\begin{subequations}
\label{eq:del}
\begin{gather}
\delta^\ep (x) = \frac1{\ep^n}\prod_{i=1}^n  \phi\left(\frac{x_i}\ep\right)
\quad (x=(x_1,\ldots,x_n));\label{eq:dela}\\
\phi\mbox{ is continuous in $\mathbb{R}$},\quad 
 \operatorname{supp}\phi\subset B(0,K\ep),\quad 
 \int_{\mathbb{R}}\phi(s)~ds=1\label{eq:delb}
\end{gather}
\end{subequations}
with $K>0$. In \eqref{eq:dela}, 
the function $(1/\ep) \phi({x_i}/{\ep})$ is an approximation of the
one-dimensional Dirac delta. Then, we can calculate as: 
\begin{subequations}
\label{eq:del2}
\begin{align}
\int_{{\mathbb R}^n} \delta^\ep(y)~dy 
&= 1;\label{eq:del2a}\\
 \int_{{\mathbb R}^n} |y|^p |\delta^\ep(y)|^p~dy& 
\le C_3\ep^{p-pn+n};\label{eq:del2b}\\
\int_{{\mathbb R}^n}  |\delta^\ep(y)|^p~dy & 
\le C_3'\ep^{-pn+n}\label{eq:del2c},
\end{align}
\end{subequations}
where $C_3$ and $C_3'$ denote positive constants depending only on $p$,
$n$, $K$ and $\|\phi\|_{L^\infty(\mathbb{R})}$. For example, if $n=3$,  
\begin{align*}
\int_{{\mathbb R}^n} & |y|^p |\delta^\ep(y)|^p~dy \\
&\le \ep^{p-pn+n} \int_0^{\sqrt n K} \int_0^{2\pi} \int_0^{\pi} 
 s^{p+n-1} |\phi(s\cos\varphi\sin\theta)|^p \cdot \\
 &\mbox{ }\qquad \cdot 
|\phi(s\sin\varphi\sin\theta)|^p |\phi(s\cos\theta)|^p 
\sin\theta~dsd\varphi d\theta \\
&\le \frac{4\pi}{p+n} (\sqrt 3 K)^{p+n} \|\phi\|_{L^\infty(\mathbb{R})}^ {3p} \ep^{p-pn+n}.
\end{align*}
Similarly, we can take  
$C_3 = \frac{2\pi}{p+2} (\sqrt 2 K)^{p+2}
\|\phi\|_{L^\infty(\mathbb{R})}^{2p}$ if $n=2$.  

Therefore, our error estimate for the regularized problem is given as
follows.

  \begin{prop}
   \label{la:err0}
  Let $1\le p<\frac{n}{n-1}$ and suppose that {\rm (A1$_p$)} is satisfied. Let $F\in L^p(\Theta)$. 
Let $(u,p)$ and $(u^\ep,p^\ep)$, respectively, be the weak solutions of \eqref{eq:1}
 and \eqref{eq:1e} with \eqref{eq:del}. Then, we have
\begin{equation}
\label{eq:17}
  \| u-u^\ep \|_{W^{1,p}} + \|p-p^\ep\|_{L^p} 
  \le C \|F\|_{L^p(\Theta)}\ep^{1-n+\frac{n}{p}},
\end{equation}
where $C$ denotes a positive constant depending only on $n$, $p$,
   $\|J_X\|_{L^\infty(\Theta)}$, $K$, $\|\phi\|_{L^\infty(\mathbb{R})}$ and
   $\Omega$. 
  \end{prop}

  \begin{remark}
   Proposition \ref{la:err0} remains valid for a bounded Lipschitz
   domain $\Omega$.  
  \end{remark}

\section{Discretization by finite element method} 
\label{sec:6}

This section is devoted to a study of the finite element approximation
applied to \eqref{eq:1e}. We introduce a family of regular triangulations
$\{\mathscr{T}_h\}_h$ of $\Omega$ (see \cite[(4.4.16)]{bs08}).
Hereinafter, we set $h=\max\{h_T\mid T\in \mathscr{T}_h\}$, where $h_T$ denotes the diameter of $T$.
For any $T\in \mathscr{T}_h$, 
let $\mathscr{P}_1(T)$ be the set of all polynomials defined on $T$ of degree $\le 1$, 
and let ${\mathscr{B}}(T)=[\mathscr{P}_1(T)
\oplus \operatorname{span}\{\lambda_{1}\lambda_{2}\cdots\lambda_{n+1}\}]^n$, where
$\lambda_{i}$ are the barycentric coordinates of $T$. Below, we consider
the P1-b/P1 element (MINI element) approximation. That is, set  
\begin{gather*}
V_h= \{v_h\in C(\overline{\Omega})^n\cap W_0^{1,2}(\Omega)^n \mid v_h|_T\in
 {\mathscr{B}}(T)\ (\forall T\in \mathscr{T}_{h})\},\\
Q_h=\{q_h\in C(\overline{\Omega})\cap L_0^2(\Omega) \mid q_h|_T\in \mathscr{P}_1(T)\ (\forall T\in \mathscr{T}_h)\}.
\end{gather*}
It is well-known that (see \cite[Lemma II.4.1]{gr86}) a pair of $V_h$ and $Q_{h}$ satisfies the uniform Babu{\v s}ka--Brezzi (inf--sup) condition
\begin{equation*}
\sup_{v_h\in V_h}\frac{b(v_h,q_h)}{\|v_h\|_{W^{1,2}}}\ge
 \beta\|q_h\|_{L^2}\qquad (q_h\in Q_{h}),
\end{equation*} 
where $\beta>0$ is independent of $h$. 

 \begin{remark}
We deal with the P1-b/P1 element only for the sake of simple
  presentation. An arbitrary pair of 
conforming finite element spaces $V_h\subset W^{1,2}_0(\Omega)^n$ and
  $Q_h\subset L_0^2(\Omega)$ satisfying  the uniform Babu{\v s}ka--Brezzi
  condition is available. 
 \label{rem:3.5}
 \end{remark}

We state the finite element approximation to \eqref{eq:1e}: 
Find $(u^\ep_h,\pi^\ep_h)\in V_h\times Q_h$ such that 
\begin{subequations}
\label{eq:21}
\begin{align}
 & a(u^\ep_h,v_h) + b(\pi^\ep_h,v_h)  
=(f^\ep,v_h)_{L^2} && (\forall v_h \in V_h), \label{eq:21a}\\
&b(q_h,u^\ep_h) = 0 && (\forall q_h \in Q_h). \label{eq:21b}
\end{align}
\end{subequations}

The finite element approximation $(w_h,r_h)\in V_h\times Q_h$ of 
\eqref{eq:w} is defined similarly. 

We make the following assumption:

\smallskip

\noindent \textbf{(A2$_p$)} For a given $g\in L^p(\Omega)^n$, the finite element
approximation $(w_h,r_h)\in V_h\times Q_h$ of \eqref{eq:w} admits
\[
 \|w-w_h\|_{W^{1,p}}+\|r-r_h\|_{L^p}
\le C_4 \inf_{(v_h,q_h)\in V_h\times Q_h}
\left( \|w-v_h\|_{W^{1,p}}+\|r-q_h\|_{L^p}\right),
\]
where $C_4$ denotes a positive constant depending only on $p$ and
$\Omega$, and $(w,r)\in W^{1,p}_0(\Omega)^n\times L^p_0(\Omega)$ the weak
solution of \eqref{eq:w}.

\begin{remark}
 \label{rem:3.7}
 If $\Omega$ is a convex polyhedral domain in $\mathbb{R}^n$ with
 $n=2,3$ and $\{\mathcal{T}_h\}_h$ is quasi-uniform (see \cite[(4.4.15)]{bs08}),
 then 
 (A2$_p$) is actually satisfied for $1<p\le\infty$; see Corollaries 4,
 5, 6 and Remark 4 of a sophisticated paper \cite{gns15}.   
However, we directly assume (A2$_p$)  instead
 of the shape condition on $\Omega$ as before.
\end{remark}

Applying the standard interpolation/projection error estimates, we obtain
the following. 

 \begin{prop}
  \label{la:err1}
  Let $1\le p <\infty$ and suppose that {\rm (A1$_p$)} and {\rm
  (A2$_p$)} are satisfied. 
Let $(u^\ep,\pi^\ep)$ and $(u^\ep_h,\pi^\ep_h)$ be solutions of \eqref{eq:1e} and \eqref{eq:21}, respectively. Then, we have
\begin{equation}
 \label{eq:dis-err}
\|u^\ep-u^\ep_h\|_{W^{1,p}}+\|\pi^\ep-\pi^\ep_h\|_{L^p} \le C h \|f^\ep\|_{L^p},
\end{equation}
where $C$ denotes a positive constant depending only on $p$ and $\Omega$.  
 \end{prop} 
 
Putting together those results, we deduce the following error estimate. 

 \begin{prop}
\label{prop:error}
 Let $1\le p<\frac{n}{n-1}$ and suppose that {\rm (A1$_p$)} and {\rm
  (A2$_p$)} are satisfied. Assume $F\in L^{p'}(\Theta)$.  
Let $(u,\pi)$ and $(u^\ep_h,\pi^\ep_h)$ be solutions of \eqref{eq:1} and
  \eqref{eq:21} with \eqref{eq:del}, respectively. Then, we have
\begin{equation}
 \label{eq:err10}
  \|u-u_h^\ep\|_{W^{1,p}}+\|\pi-\pi_h^\ep\|_{L^p} \le 
C\ep^{-n+\frac{n}{p}}(\ep +h),
\end{equation}
where $C$ denotes a positive constant depending only on $n$, $p$,
  $\|J_X\|_{L^\infty(\Theta)}$, $\operatorname{meas}(\Theta)$, $K$,
  $\|\phi\|_{L^\infty(\mathbb{R})}$, $\|F\|_{L^p(\Theta)}$,
  $\|F\|_{L^{p'}(\Theta)}$ and $\Omega$.  
 \end{prop}

 \begin{proof}
Since $f^\ep$ is defined in terms of $\delta^\ep$ given by
  \eqref{eq:del}, 
  we have by \eqref{eq:del2c}
  \begin{align*}
   \|f^\ep\|_{L^p}
   &\le
   \operatorname{meas}(\Theta)^{1/p}\|F\|_{L^{p'}(\Theta)}\|\delta^\ep\|_{L^p(\mathbb{R}^n)}\\
   & \le C\operatorname{meas}(\Theta)^{1/p}\|F\|_{L^{p'}(\Theta)}\ep^{-n+\frac{n}{p}}, 
  \end{align*}
  where $C>0$ is a constant depending only on $p$, $n$, $K$, and 
  $\|\phi\|_{L^\infty(\mathbb{R})}$.   
Hence, in view of Lemmas \ref{la:err0} and \ref{la:err1},  
\begin{align*}
\|u-u^\ep_h\|_{W^{1,p}}+\|\pi-\pi^\ep_h\|_{L^p} 
&\le \|u-u^\ep\|_{W^{1,p}}+\|\pi-\pi^\ep\|_{L^p} \\
& \qquad + \|u^\ep-u^\ep_h\|_{W^{1,p}}+\|\pi^\ep-\pi^\ep_h\|_{L^p} \\
&\le C\ep^{1-n+\frac{n}{p}}+ Ch\cdot C \ep^{-n+\frac{n}{p}}.
\end{align*}
\end{proof}

We usually take as $\ep=h$ in the immersed boundary method. Therefore,
applying Proposition \ref{prop:error} with $p=1$, we obtain the optimal
order error estimate
\begin{equation}
 \label{eq:opt1}
  \|u-u^\ep_h\|_{W^{1,1}}+\|\pi-\pi^\ep_h\|_{L^1} \le 
Ch.
\end{equation}
It should be kept in mind that this estimate is available only if {\rm (A1$_p$)} and {\rm
  (A2$_p$)} are true. However, the case $p=1$ is excluded both in
  \cite{mr07} and \cite{gns15} (see Remarks \ref{rem:2.10} and
  \ref{rem:3.7}). In conclusion, we offer the following theorem as the
  final error estimate in this paper.

  \begin{thm}
\label{thm:error}
Suppose that $\Omega$ is a convex polyhedral domain in $\mathbb{R}^n$
   with $n=2,3$. Assume that $\{\mathcal{T}_h\}_h$ is a family of
   quasi-uniform triangulations. Let
   $F\in L^{\infty}(\Theta)$. Let $(u,\pi)$ and $(u^\ep_h,\pi^\ep_h)$ be solutions of \eqref{eq:1} and
  \eqref{eq:21} with \eqref{eq:del}, respectively. Further, let
   $\ep=\gamma_1h$ with a positive constant $\gamma_1$. Then, for any
   $0<\alpha<1$, there exists a positive constant $C$ depending only on
   $\gamma_1$, $n$, $\alpha$, $\Omega$,
  $K$, $\|\phi\|_{L^\infty(\mathbb{R})}$, $\|J_X\|_{L^\infty(\Theta)}$, 
  $\operatorname{meas}(\Theta)$, and $\|F\|_{L^\infty(\Theta)}$ such
   that  
\begin{equation}
 \label{eq:opt2}
  \|u-u^\ep_h\|_{W^{1,q}}+\|\pi-\pi^\ep_h\|_{L^q} \le 
Ch^{1-\alpha}\quad\mbox{with any }1\le q\le \frac{n}{n-\alpha}
\end{equation}
   and
\begin{equation}
 \label{eq:opt3}
  \|u-u^\ep_h\|_{L^r}\le 
Ch^{1-\alpha}\quad\mbox{with}\quad r=\frac{n}{n-\alpha -1}.
\end{equation}
  \end{thm}

 \begin{proof}
  As was pointed out in Remarks \ref{rem:2.10} and
  \ref{rem:3.7}, {\rm (A1$_p$)} and {\rm
  (A2$_p$)} are true for a convex polyhedral domain. 
 Setting $\alpha=n(1-1/p)$ in \eqref{eq:err10} and applying an obvious
 inequality $\|\psi\|_{L^q}\le C\|\psi\|_{L^p}$ for $1\le q\le n/(n-\alpha)$, we deduce
 \eqref{eq:opt2}. Inequality \eqref{eq:opt3} is a consequence of
 \eqref{eq:err10} and the Sobolev embedding theorem (see \cite[Theorem
 4.12, Part I, Case C]{af03}). 
 \end{proof}

\begin{remark}
 The exponent $r$ in \eqref{eq:opt3} is included in $2<r<\infty$ if
 $n=2$ and in $3/2<r<3$ if $n=3$. 
 \label{rem:3.16}
\end{remark}



\section{Numerical integration}
\label{sec:7}

In this section, we study the error caused
   by numerical integrations for computing $f^\ep$. As will be stated
   below, a simple numerical integration formula does not spoil the
   accuracy of the immersed boundary method described in Theorem \ref{thm:error}. 
    
First, we deal with the case $n=2$.
Suppose that we are given a continuous function $F(\theta)$ in
$\Theta=(c_1,d_1)$ with $c_1<d_1$. Let us introduce a partition
$c_1=\theta_0<\theta_1<\cdots<\theta_{M}=d_1$. Moreover, letting
$\theta_{-\frac12}=\theta_0$, 
$\theta_{i-\frac12}=(\theta_i+\theta_{i-1})/2$ for $1\le i\le M$, and
$\theta_{M+\frac12}=\theta_M$, we set
$\zeta_i=\theta_{i+\frac12}-\theta_{i-\frac12}$ for $0\le i\le M$.
 Further, set ${\zeta=\max_{0\le i\le
M}\zeta_i}$.
Then, we employ the midpoint rule to compute $f^\ep$, that is, 
\begin{equation}
\label{eq:71a}
 f^{\ep, \zeta} (x)
 = \sum_{i=0}^M F(\theta_i) \delta^\ep_{X(\theta_i)}(x)
\zeta_i 
= \sum_{i=0}^M F(\theta_i) \delta^\ep(x-X(\theta_i))
\zeta_i. 
\end{equation}
It is useful to express $f^{\ep,\zeta}$ as
\begin{equation}
\label{eq:71b}
 f^{\ep, \zeta} (x)
= \int_{\Theta} \hat{F}^\zeta(\theta) \delta^\ep(x-\hat{X}^\zeta(\theta))~d\theta, 
\end{equation}
where
$\hat{F}^\zeta(\theta)$ and 
$\hat{X}^\zeta(\theta)=(\hat{X}_1^\zeta(\theta),\hat{X}_2^\zeta(\theta))$
are piecewise constant functions such
that
\[
 \hat{F}^\zeta(\theta)=F(\theta_i),\quad
 \hat{X}^\zeta(\theta)=X(\theta_i)\qquad (\theta_{i-\frac12}<
 \theta\le \theta_{i+\frac12},\ 0\le i\le M).
\]
From the standard theory, we know
\begin{subequations}
 \label{eq:74}
  \begin{align}
  \|F-\hat{F}^\zeta\|_{L^\infty(\Theta)}&\le C\zeta
  |F|_{W^{1,\infty}(\Theta)},\label{eq:74a}\\
  \|X-\hat{X}^\zeta\|_{L^\infty(\Theta)^n}&\le C\zeta
  |X|_{W^{1,\infty}(\Theta)^n}, \label{eq:74b}
  \end{align}
  \end{subequations}
where $|F|_{W^{1,\infty}(\Theta)}$ denotes the seminorm in $W^{1,\infty}(\Theta)$ for example. 

For the case of $n=3$, $f^{\ep,\zeta}$ is defined similary. We introduce a suitable partition of
$\Theta=(c_1,d_1)\times (c_2,d_2)$ with $c_l<d_l$, $l=1,2$, and the
size parameter $\zeta>0$. Let $\hat{F}^\zeta(\theta)$ and 
$\hat{X}^\zeta(\theta)$ be piecewise
constant interpolations of $F(\theta)$ and $X(\theta)$,
respectively. Then, $f^\ep$ is approximated by $f^{\ep,\zeta}$ defined as \eqref{eq:71b}. For the partition of
$\Theta$, we only assume so that \eqref{eq:74} hold true.

\begin{remark}
Let $n=2$. 
If $F(\theta)$ is a periodic function, $F(c_1)=F(d_1)$, and the
partition is uniform, $f^{\ep,\zeta}$ is coincide with the trapezoidal
rule for $F(\theta) \delta^\ep(x-X(\theta))$. However, we here do not
 explicitly assume
 the periodicity for $F(\theta)$.  
\end{remark}

\begin{prop}
  \label{la:reg-err2}
Let $\delta^\ep$ be a continuous function satisfying
  \eqref{eq:dep}. Suppose that $F\in C^1(\Theta)$. 
  Then, for $1\le p<\frac{n}{n-1}$, we have
\[
  \|f-f^{\ep,\zeta}\|_{W^{-1,p}}\le C 
\left[ \left| 1- \int_{{\mathbb R}^n} \delta^\ep (y) ~dy \right|
+  \| \rho \delta^\ep\|_{L^p({\mathbb R}^n)} +\zeta+\zeta^{1-n+\frac{n}{p}}
  \right],
\]
where $\rho(x)=x$ and $C$ denotes a positive constant depending only
  on $n$, $p$, $\|J_X\|_{L^\infty(\Theta)}$, $|F|_{W^{1,\infty}(\Theta)}$ and $|X|_{W^{1,\infty}(\Theta)^n}$. 
\end{prop}

\begin{proof}
It is a just modification of the proof of Proposition
 \ref{la:reg-err}. 
 Let $\varphi \in C^\infty_0(\Omega)^n$ and express it as
  \[
   \varphi(x) = \varphi(\hat{X}^\zeta(\theta)) + (x-\hat{X}^\zeta(\theta))\cdot 
  \int_0^1 \nabla \varphi( t( x-\hat{X}^\zeta(\theta) ) +
 \hat{X}^\zeta(\theta) ) ~dt
  \]
for $x\in\mathbb{R}^n$. Using this, we have 
\begin{multline*}
  \dual{f-f^{\ep,\zeta}, \varphi}_{W^{-1,p},W_0^{1,p}}\\
 =
 \underbrace{
 \int_\Theta F(\theta) \varphi(X(\theta))~d\theta
-\int_\Theta F(\theta) \varphi(\hat{X}^\zeta(\theta))\left(\int_\Omega
\delta^\ep (x-\hat{X}^\zeta(\theta)) ~dx\right) d\theta}_{=I_1}
 \\
\underbrace{-\int_0^1\int_\Theta \hat{F}^\zeta(\theta) \int_\Omega \delta_{\hat{X}^\zeta(\theta)}^\ep (x)
(x-\hat{X}^\zeta(\theta))\cdot\nabla \varphi( t( x-\hat{X}^\zeta(\theta) ) + \hat{X}^\zeta(\theta) ) ~dxd\theta dt}_{=I_2}. 
\end{multline*}
 
To estimate $|I_1|$, we further divide it as follows: 
\begin{multline*}
I_1=
 \underbrace{\int_\Theta [F(\theta)-\hat{F}^\zeta(\theta)] \varphi(X(\theta))~d\theta}_{=I_{11}}
+\underbrace{\int_\Theta \hat{F}^\zeta(\theta)[ \varphi(X(\theta))- \varphi(\hat{X}^\zeta(\theta))]~d\theta}_{=I_{12}}\\
+\underbrace{\int_\Theta \hat{F}^\zeta(\theta)
 \varphi(\hat{X}^\zeta(\theta))\left(1-\int_\Omega\delta^\ep
 (x-\hat{X}^\zeta(\theta)) ~dx\right) d\theta}_{=I_{13}}. 
\end{multline*}

As in the proof of Proposition
 \ref{la:reg-err}, we derive
\[
 |I_{13}|\le C \|\hat{F}^\zeta\|_{L^p(\Theta)}
 \left| 1- \int_{{\mathbb R}^n} \delta^\ep (y) ~dy \right|
 \|\varphi\|_{W^{1,{p'}}(\Omega)}.
\]
By \eqref{eq:74}, we have 
\begin{align*}
 |I_{11}|
 &\le C\|F-\hat{F}^\zeta\|_{L^p(\Theta)}\|\varphi\|_{L^{p'}(\Theta)}\\
 &\le C\zeta|F|_{W^{1,\infty}(\Theta)}\|\varphi\|_{W^{1,p'}(\Omega)}.
\end{align*}
We apply Morrey's inequality to obtain 
\begin{align*}
 |I_{12}|
 &\le C\int_\Theta |\hat{F}^\zeta(\theta)|\cdot |X(\theta)-
 \hat{X}^\zeta(\theta)|^{1-n/p'}\cdot \|\varphi\|_{W^{1,p'}}~d\theta\\
 &\le C\zeta^{1-n/p'}\|\hat{F}^\zeta\|_{L^1(\Theta)} |X|_{W^{1,\infty}(\Theta)^n}\|\varphi\|_{W^{1,p'}}.
\end{align*}
Estimation for $|I_2|$ is done in exactly the way as the proof of Proposition
 \ref{la:reg-err}, that is, we deduce
\[
 |I_2|
 \le C\|\hat{F}^\zeta\|_{L^1(\Theta)} \| \rho\delta^\ep
 \|_{L^p({\mathbb R}^n)}\|{\varphi}\|_{W^{1,{p'}}}.
 \]
Finally, noting $\|\hat{F}^\zeta\|_{L^1(\Theta)}\le
 C\|F\|_{L^\infty(\Theta)}$, we get the disired inequality. 
\end{proof}

Applying 
Proposition \ref{la:reg-err2}
instead of Proposition \ref{la:reg-err}, we obtain the following
result.

  \begin{thm}
   \label{thm:error1}
   Let $(u,\pi)$ and $(u^\ep_h,\pi^\ep_h)$ be solutions of \eqref{eq:1} and
   \eqref{eq:21} with \eqref{eq:del}, respectively, where $f^\ep$ is
   replaced by $f^{\ep,\zeta}$ defined as \eqref{eq:71b}. In addition to
   assumptions of Theorem \ref{thm:error}, we assume that $F\in
   C^{1}(\Theta)$.
   Further, let
   $\zeta=\gamma_2h$ with a positive constant $\gamma_2$. Then, error
   estimates \eqref{eq:opt2} and \eqref{eq:opt3} remain true. 
 \end{thm}

\section{Numerical experiments}
\label{sec:8}

Throughout this section, we let $\Omega = (-1, 1)^2 \subset
{\mathbb{R}^2} $ and $\Gamma=B(0,1/2)$.
We consider the stationary Stokes problem
\begin{equation*}
 -\nu \Delta u +\nabla \pi  = f + g \mbox{ in }\Omega, \quad
 \nabla \cdot u =0 \mbox{ in } \Omega,
 \quad u=0 \mbox{ on }  \partial \Omega.
\end{equation*}
Herein, we have added an extra outer force
$g = (1 , 0)$ in order to illustrate a pressure jump across $\Gamma$.  
We also set $\Theta = [0, 2\pi]$ and
\begin{equation*}
X(\theta) = \frac12 (\cos (\theta), \sin (\theta) ). 
\end{equation*}
In accordance to the simplest elasticity modeling (see
\cite{bgh07,pes02}), we take 
$F(\theta) = \kappa {\partial^2 X}/{ \partial \theta^2}$ 
with $\kappa$ is a suitable positive constant. Specifically, taking
$\kappa=2$, we set 
\begin{equation*}
F(\theta) = - \left(  \cos (\theta),  \sin (\theta) \right).
\end{equation*}


We deal with the following problem: Find $(u^\ep_h,p^\ep_h)\in V_h\times Q_h$ such that 
\begin{align*}
 & a(u^\ep_h,v_h) + b(p^\ep_h,v_h)  
=(f^{\ep,\zeta}+g,v_h)_{L^2} && (\forall v_h \in V_h), \\
&b(q_h,u^\ep_h) = 0 && (\forall q \in Q_h). 
\end{align*}
Herein, $f^{\ep,\zeta}$ is defined as \eqref{eq:71b}  
and $\zeta ={2 \pi}/{M}$, $\theta_i = i\zeta$, $0\le i \le M$. 
As a choice of $\delta^\ep$, we examine the following discrete delta function 
\begin{equation*}
\phi(s)= 
\begin{cases}
 \frac1{2} \left( 1+ \cos( \pi s) \right) &( |s| \le 1)\\
0 &(\mbox{otherwise}).
\end{cases}
\end{equation*}

For discretization, we take $\mathscr{T}_h$ as a uniform mesh composed of $2N^2$
congruent right-angle triangles; 
each side of $\Omega$ is divided into $N$ intervals of
the same length. Then each small square is decomposed into two equal 
triangles by a diagonal. Each parameters are set as follows:
\begin{equation*} 
h = \frac{\sqrt2}{N}, \quad \ep = h, \quad M=N,\quad \mbox{and} \quad \zeta = \frac{2\pi}{M}=\sqrt2\pi \ep.
\end{equation*}


To confirm convergence results described in Theorems \ref{thm:error} and \ref{thm:error1}, we compute the following quantities:
\begin{equation*}
E_h^{r(0)} = \|\tilde u-u_h\|_{L^r}, \quad 
E_h^{r(1)} = \|\tilde u - u_h\|_{W^{1,r}},  \quad
\mbox{and} \quad E_h^{r(3)} = \|\tilde \pi - \pi_h\|_{L^r},  \quad
\end{equation*}
where $(\tilde u, \tilde \pi)\in V_{h'} \times Q_{h'}$ denotes the
numerical solution using a finer triangulation
$\mathcal{T}_{h'}$. Moreover, we compute 
\begin{equation*}
\rho^{r(i)}_h = \frac{\log E_{2h}^{r(i)} - \log E_h^{r(i)}}{\log2h-\log h} \quad (i=1,2,3) .
\end{equation*}

The result is reported in Table \ref{tab:tab1}--\ref{tab:tab3}. 
We observe from Table \ref{tab:tab1} that convergence rates of the
$W^{1,1}$-error for velocity is first-order 
while that of the $L^1$-error for pressure is larger than $1$. 
It is also observed that as $p$ becomes larger, each convergence rate
becomes worse.
Nevertheless, the rate of the $L^2$-error for velocity is still larger
than $1$; see Table \ref{tab:tab3}.   
All of those numerical results support our theoretical results. 
From those numerical observations, we infer that the following
optimal-order error estimate,  
\begin{equation*}
  \|u^\ep-u^\ep_h\|_{W^{1,r}}+h \|u^\ep-u^\ep_h\|_{L^r}  \le 
   Ch^{1-\alpha}
\end{equation*}
actually holds true. However, we postpone the proof of this conjecture
for future study. 

 \begin{figure}[htb]
  \begin{center}
\includegraphics[width=.3\textwidth]{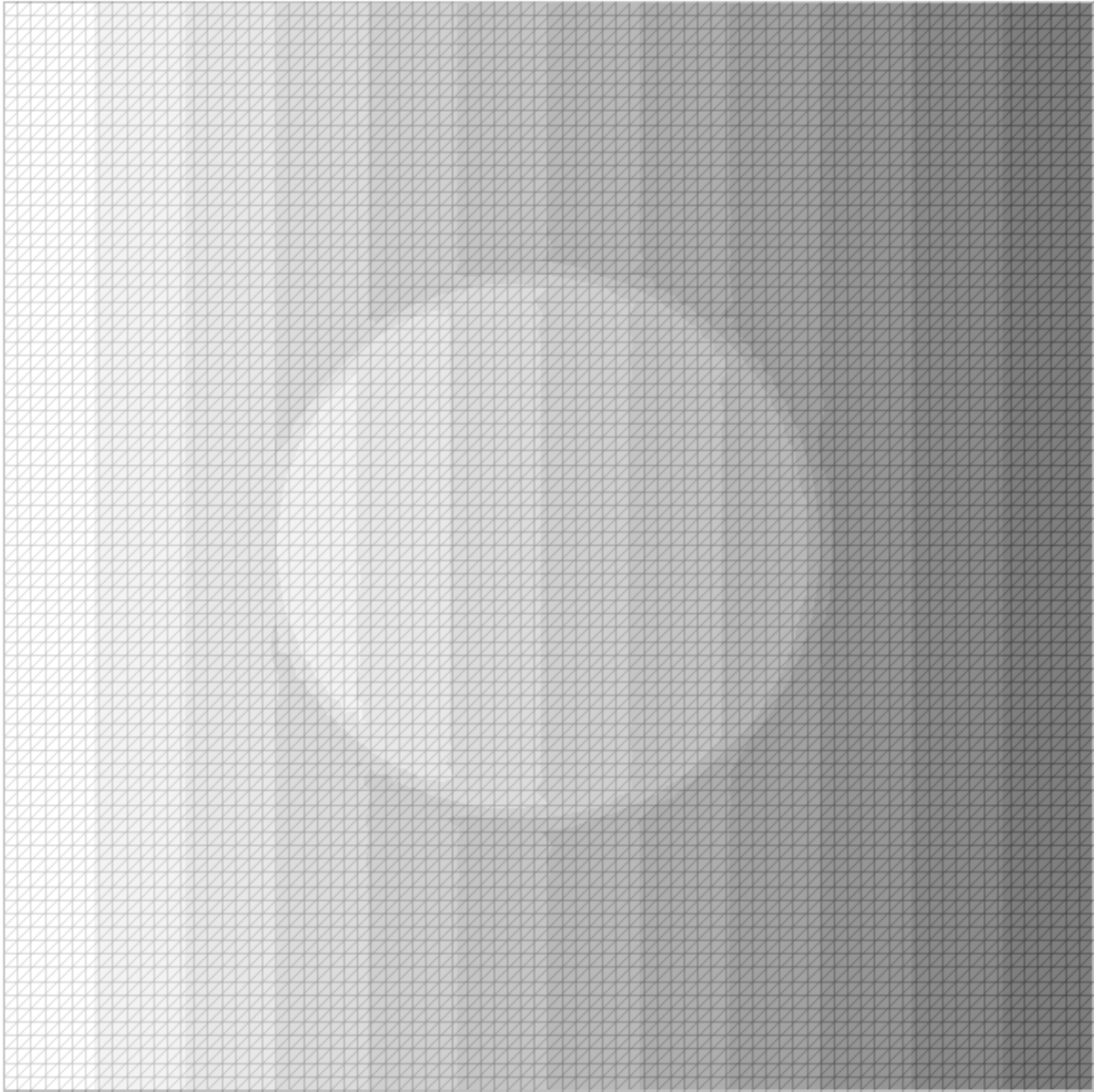}   
  \end{center}
\caption{Profile of pressure $\pi^\ep_h$ for $N=80$.  A jump of pressure
  is observed across $\Gamma$ so that $\pi^\ep_h$ becomes a
  discontinuous function.
}
\label{fig:jump}
 \end{figure}

\begin{table}[htb]
\begin{tabular}{l|lr|lr|lr}
$h$ & $E_h^{1(1)}$ & $\rho_h^{1(1)}$ & $E_h^{1(2)}$  & $\rho_h^{1(2)}$  & $E_h^{1(3)}$ & $\rho_h^{1(3)}$ \\
 \hline
0.2828 & 0.000436582 & --- & 0.0104959 & --- & 0.0508396 & --- \\
0.1414 & 0.00010817 & 2.0129 & 0.00525413 & 0.9983 & 0.0195045 & 1.382 \\
0.0707 & 2.61239e-05 & 2.0498 & 0.00262956 & 0.9986 & 0.00892026 & 1.128 \\
0.0353 & 7.315e-06 & 1.8364 & 0.0012387 & 1.086 & 0.00269059 & 1.729 \\
\end{tabular}
\caption{Convergence rates in $W^{1,r}\times L^r$ with $r=1$. }
\label{tab:tab1}
\end{table}

\begin{table}[htb]
\begin{tabular}{l|lr|lr|lr}
$h$ & $E_h^{1.5(1)}$ & $\rho_h^{1.5(1)}$ & $E_h^{1.5(2)}$  & $\rho_h^{1.5(2)}$  & $E_h^{1.5(3)}$ & $\rho_h^{1.5(3)}$ \\ 
\hline
0.2828 & 5.92276e-06 & --- & 0.000648725 & --- & 0.00932687 & --- \\
0.1414 & 9.07843e-07 & 1.803 & 0.000298948 & 0.745 & 0.00291411 & 1.118 \\
0.0707 & 1.51842e-07 & 1.719 & 0.000150538 & 0.659 & 0.00113506 & 0.906 \\
0.0353 & 2.37657e-08 & 1.783 & 7.14263e-05 & 0.717 & 0.000222016 & 1.569 \\
\end{tabular}
\caption{Convergence rates in $W^{1,r}\times L^r$ with $r=3/2$.}
\label{tab:tab2}
\end{table}

\begin{table}[htb]
\begin{tabular}{l|lr|lr|lr}
$h$ & $E_h^{2(1)}$ & $\rho_h^{2(1)}$ & $E_h^{2(2)}$  & $\rho_h^{2(2)}$  & $E_h^{2(3)}$ & $\rho_h^{2(3)}$ \\
 \hline
0.2828 & 8.89239e-08 & --- & 4.5878e-05 & --- & 0.00196196 & --- \\
0.1414 & 8.78242e-09 & 1.669 & 1.98618e-05 & 0.60390 & 0.000535873 & 0.9361 \\
0.0707 & 1.03983e-09 & 1.539 & 1.02922e-05 & 0.47422 & 0.000176548 & 0.8009 \\
0.0353 & 1.05306e-10 & 1.651 & 5.33301e-06 & 0.47426 & 2.25745e-05 & 1.4836 \\
\end{tabular}
\caption{Convergence rates in $W^{1,r}\times L^r$ with $r=2$.}
\label{tab:tab3}
\end{table}

\section*{Acknowledgement}
NS is supported by CREST, Japan Science
and Technology Agency and JSPS KAKENHI Grant Number 15H03635, 15K13454. 



\bibliographystyle{plain}
\bibliography{ibm1} 

\end{document}